\documentclass[11pt]{article}
\usepackage[latin1]{inputenc}

\usepackage{epsfig}
\usepackage{mathrsfs}
\usepackage{color}
\usepackage[british,english]{babel}
\usepackage{amsthm}
\usepackage{amsmath}
\usepackage{amsfonts}
\usepackage{amssymb}
\usepackage{graphicx}
\newtheorem{corollary}{Corollary}[section]

\newtheorem{lemma}[corollary]{Lemma}

\newtheorem{theorem}[corollary]{Theorem}
\newfont{\sBlackboard}{msbm10 scaled 900}

\newcommand{\dd}     {{\rm d}}
\newcommand{\mylabel}[1]{\label{#1}
            \ifx\undefined\stillediting
            \else \fbox{$#1$}\fi }
\newcommand{\BE}{\begin{equation}}

\newcommand{\EEQ}{\end{equation}}
\newcommand{\rfb}[1]{\mbox{\rm
   (\ref{#1})}\ifx\undefined\stillediting\else:\fbox{$#1$}\fi}

\newfont{\Blackboard}{msbm10 scaled 1200}
\newcommand{\bl}[1]{\mbox{\Blackboard #1}}
\newfont{\roma}{cmr10 scaled 1200}

\def\CC{\rm \hbox{C\kern-.56em\raise.4ex
         \hbox{$\scriptscriptstyle |$}\kern+0.5 em }}

\newcommand{\nline}  {{\bl N}}
\newcommand{\rline}  {{\mathbb R}}

\def\cD{{\cal D}}

\newcommand{\mm}    {{\hbox{\hskip 0.5pt}}}

\newcommand{\bluff} {{\hbox{\raise 15pt \hbox{\mm}}}}
\makeatletter
\def\section{\@startsection {section}{1}{\z@}{-3.5ex plus -1ex minus
    -.2ex}{2.3ex plus .2ex}{\large\bf}}
\makeatother
\def\be{\begin{equation}}
\def\ee{\end{equation}}
\def\beqs{\begin{eqnarray*}}
\def\eeqs{\end{eqnarray*}}

\def\ds{\displaystyle}
\renewcommand{\div}     {{\rm div\,}}

\begin{document}
\thispagestyle{empty}
\title{\bf Polynomial stabilization of some dissipative hyperbolic systems}
\author{K. Ammari \thanks{D\'epartement de Math\'ematiques,
Facult\'e des Sciences de Monastir, Universit\'e de Monastir, 5019 Monastir, Tunisie,
e-mail~:kais.ammari@fsm.rnu.tn} \, E. Feireisl
\thanks{Institute of Mathematics of the Academy of Sciences of the Czech Republic, $\hat{\hbox{Z}}$itn\`a 25, 115 67 Praha 1, Czech Republic, e-mail~: feireisl@math.cas.cz. Supported by Grant 201/09/ 0917 of GA \v CR as a part of the general research programme of the Academy
of Sciences of the Czech Republic, Institutional Research Plan RVO: 67985840.}\, and \, S. Nicaise
\thanks{Universit\'e de Valenciennes et du Hainaut Cambr\'esis,
LAMAV, FR CNRS 2956, Le Mont Houy, 59313 Valenciennes Cedex 9,
France, \, email~: snicaise@univ-valenciennes.fr}
}
\date{}
\maketitle
{\bf Abstract.}{ {\small
We study the problem of stabilization for the acoustic system with a spatially distributed damping. Imposing various
hypotheses on the structural properties of the damping term, we identify either exponential or polynomial
decay of solutions with growing time. Exponential decay rate is shown by means of a time domain approach, reducing the problem
to an observability inequality to be verified for solutions of the associated
conservative problem. In addition, we show a polynomial stabilization result, where
the proof uses a frequency domain method and combines a contradiction argument
with the multiplier technique to carry out a special analysis for
the resolvent.
}}

\section{Introduction} \label{intro}

{{We consider the following system of equations:}}

\be
\label{fluide}
\left\{
\begin{array}{l}
\vec{u}_t + \nabla r + \alpha \, \vec{u} = 0, \,\hbox{ in } \Omega \times \rline^+, \\
r_t + div \vec{u} = 0, \,\hbox{ in } \Omega \times \rline^+, \\
\vec{u} \cdot n = 0, \,\hbox{ on } \Gamma \times \rline^+, \\
\vec{u}(0,x) = \vec{u}^0(x), \, r(0,x) = r^0(x), \, x \in \Omega,
\end{array}
\right.
\ee
where $\Omega$ is a  bounded {{domain in}}  $\rline^d, \, d=2,3$,
with a smooth boundary $\Gamma$, $div = \nabla\cdot$ is the divergence operator and {{ $\alpha \in \mathcal{C}^\infty (\overline{\Omega})$}},
with $\alpha \ge 0$ on $\Omega$ and {{such that}}
\begin{equation}\label{eq:a}
\exists\ \alpha_->0 \mbox{ such that } \alpha \ge \alpha_-
\mbox{ on } \omega.
\end{equation}
Here $\omega \ne \emptyset$ stands for the {{subset}} of $\Omega$ on which the feedback is active.
As usual $n$ denotes the unit outward normal vector along $\Gamma$.

{{ The system of equations (\ref{fluide}) is a linearization of the \emph{acoustic equation} governing the propagation of 
acoustic waves in a compressible medium, see Lighthill \cite{Lighthill:78,Lighthill:52,Lighthill:54}, where $\alpha \vec{u}$ represents a damping term of Brinkman type. 
This kind of damping arises also in the process of homogenization (see Allaire \cite{Allaire:91}), and is frequently used as a suitable \emph{penalization} 
in fluid mechanics models, see Angot, Bruneau, and Fabrie \cite{AngotBruneauFabrie:99}. Our main goal is to find sufficient condition on the initial data and the function $\alpha$ so that the solution of (\ref{fluide}) stabilizes to zero for $t \to \infty$ and, if this occurs, we are interested in the rate of this decay.
}}


Let $L^2(\Omega)$ denote the standard Hilbert space of square integrable functions in $\Omega$.
To avoid abuse of notation,
we shall write $\|\cdot\|$ for the $L^2(\Omega)$-norm or the $L^2(\Omega)^d$-norm.
{{ Denoting  $H = (L^2(\Omega))^d \times L^2(\Omega)$, 
we introduce the operator
$$
{\cal A} = \left(
\begin{array}{ll}
0  & \nabla \\
div  & 0
\end{array}
\right) : {\cal D}({\cal A}) = \left\{(\vec{u},r) \in H, \, (\nabla r,div \vec{u}) \in H, \, \vec{u} . n_{|\Gamma} = 0 \right\} \subset  H \rightarrow H,
$$
and
$$
{\cal B} =  \left( \begin{array}{ll} \sqrt{\alpha} \\ 0 \end{array} \right) \in {\cal L}((L^2(\Omega))^d,H), \,
{\cal B}^* = \left( \begin{array}{cc} \sqrt{\alpha} & 0 \end{array} \right) \in {\cal L}(H, (L^2(\Omega))^d).
$$

}}

{{ 
Accordingly, the problem (\ref{fluide}) can be recast in an abstract form:
\be
\label{cauchy}
\left\{
\begin{array}{l}
\vec{Z}_t (t) + {\cal A} \vec{Z}(t) + {\cal B} {\cal B}^* \vec{Z}(t) = 0, \, t > 0, \\
\vec{Z}(0) = \vec{Z}^0,
\end{array}
\right.
\ee
where $\vec{Z} = ( \vec{u}, r )$,
or, equivalently, 
\be
\label{cauchybis}
\left\{
\begin{array}{l}
\vec{Z}_t (t) = {\cal A}_d \vec{Z}(t), \, t > 0, \\
\vec{Z}(0) = \vec{Z}^0,
\end{array}
\right.
\ee
with ${\cal A}_d = - {\cal A} - {\cal B}{\cal B}^*$ with ${\cal D}({\cal A}_d) = {\cal D}({\cal A}).$
}}

{{It can be shown (see Section \ref{section2} below) that for any initial data $(\vec{u}^0, r) \in {\cal D}({\cal A})$ the 
problem \eqref{fluide} admits a unique solution
$$(\vec{u},r) \in C([0,\infty); {\cal D}({\cal A})) \cap C^1([0, \infty); H).$$ Moreover, the solution $(\vec{u},r)$ satisfies,
the energy identity
\be
\label{energyid}
E(0) - E(t) =
 \int_0^t
\left\|\sqrt{\alpha} \, \vec{u}(s)\right\|_{(L^2(\Omega))^d}^2 \dd s \ \mbox{for all}\ t \geq 0
\ee
with
\be
\label{energy}
E(t) = \frac{1}{2} \,\left\|(\vec{u}(t),r(t)) \right\|^2_{H}, \, \forall \, t \geq 0, 
\ee
where we have denoted
$$
\left\langle (\vec{u},r), (\vec{v},p)\right\rangle_H = \int_\Omega \left( \vec{u}(x) . \vec{v}(x) + r(x) p(x) \right) \, dx , \,  \left\| (\vec{u},r)\right\|_H = \sqrt{\int_\Omega \left(\left| \vec{u}(x)\right|^2 + r^2(x) \right) \, dx}.
$$

}}

{{ Using (\ref{energy}) and the standard density arguments, we can extend the solution operator for the data 
$(\vec{u}^0, r) \in H$. Consequently, we associate to the problem (\ref{fluide}) (or to the abstract Cauchy problems (\ref{cauchy}), 
(\ref{cauchybis})) a solution (semi)-group that is globally bounded in $H$. 
}}

{{ As the energy $E$ is nonincreasing along trajectories, we want to determine the set of initial data $(\vec{u}^0, r^0)$ for which 
\begin{equation}\label{stab}
E(t) \to 0 \ \mbox{as} \ t \to \infty.
\end{equation}
Such a question is of course intimately related to the structural properties of the function $\alpha$, notably to the geometry of the 
set $\omega$ on which the damping is effective. 
}}

In this paper, {{we characterize the set of initial data for which (\ref{stab}) holds in terms of the set $\omega$, and, eventually 
we obtain some information on the rate of decay. In particular, we establish an observability inequality for the associated conservative system yielding \emph{exponential} decay and use 
a frequency domain method, combined 
with the multiplier technique, to obtain \emph{polynomial} rate of decay. It is worth-observing that the associated conservative system coincides 
with the standard linear \emph{wave} equation, supplemented with the Neumann boundary conditions, where the asymptotic behavior of solutions is relatively well understood.
}}

The paper is organized as follows. Section \ref{section2} summarizes some well known facts concerning the acoustic system \rfb{fluide}. In section \ref{section3}, we examine the spectral properties of the generator ${\cal A}_d$ and establish a strong stability results. Section \ref{section4} addresses the exponential and not exponential stability results.
In Section \ref{section5}, we prove exponential stability for a modified 
system, with a slightly different damping law, by using an observability strategy.
Polynomial stability of the modified system
is studied in Section \ref{section6}.

\section{Preliminaries}
\label{section2}

We start with a simple observation that the problem (\ref{fluide}) can be viewed as a \emph{bounded} (in $H$) perturbation of the conservative system 
\be
\label{fluidec}
\left\{
\begin{array}{l}
\vec{u}_t + \nabla r  = 0, \,\hbox{ in } \Omega \times \rline^+, \\
r_t + div \vec{u} = 0, \,\hbox{ in } \Omega \times \rline^+, 
\end{array}
\right\}
\ee
which can be recast as the standard \emph{wave equation}
\[
r_{t,t} - \Delta r = 0.
\]
Consequently, the basic existence theory for (\ref{fluide}) derives from that of (\ref{fluidec}).

Moreover, since the boundary $\Gamma$ as well as the damping coefficient $\alpha$ are smooth, solutions of (\ref{fluide})
remain smooth as soon as we take $\vec{u}^0$, $r^0$ smooth and satisfying relevant compatibility conditions as the case may be. In what follows, we may therefore deal with smooth solution, whereas the results for  data in $H$ can be obtained by means of density arguments.

\subsection{Long-time behavior} 

The operator ${\cal A}_d$ possesses a non-trivial kernel that is left invariant by the evolution, namely, solutions of the ``stationary'' problem 
\begin{equation}\label{st1}
\nabla r + \alpha \vec{u} = 0, \ \nabla \cdot \vec{u} = 0, \ \vec{u} \cdot n|_{\Gamma} = 0.
\end{equation} 
Thus the stationary field $\vec u$ is solenoidal, and, integrating (\ref{st1}) over $\Omega$ yields 
\[
\vec u = 0 \ \mbox{in}\ {\rm supp}\ \alpha , \ \nabla r = 0.
\]
Accordingly, we introduce the space 
\[
E = {\rm Ker}[{\cal A}_d] = \{ (\vec u,r) \ | \ \nabla \cdot \vec u = 0, \ \vec u|_{{\rm supp} \alpha} = 0, 
\ \vec u \cdot \vec{n}|_{\Gamma} = 0, \ r = {\rm const} \},
\]
together with its orthogonal complement (in $H$) denoted $H_0$. 

It is easy to check that 
\[
\left< {\cal A}_d ( \vec w, s ), (\vec u, r) \right>_H = 0 \ \mbox{for any}\ 
(\vec w, s) \in {\cal D}({\cal A}), \ (\vec u,r) \in E;
\] 
in particular, the solution operator associated to (\ref{fluide}) leaves both $E$ and $H_0$ invariant. 
Consequently, the decay property (\ref{stab}) may hold only for the initial data emenating from the set $H_0$.

\section{Strong stability}\label{section3}

The following observation can be shown by a simple density argument:

\begin{lemma}
The solution $(\vec{u},r)$ of \rfb{fluide} with
initial datum in ${D}(\mathcal{A}_d)$ satisfies
\be
\label{deriveeenergy}
E'(t) = - \int_\Omega \alpha\left|\vec{u}\right|^2dx\leq0.
\ee
Therefore the energy is non-increasing and \rfb{energyid} holds for all initial datum in $H$.
\end{lemma}

As already shown in the previous section, the strong stability result (\ref{stab}) may hold only if 
we take the initial data
\[
(\vec u^0, r^0) \in H_0 = {\rm Ker}[{\cal A}_d]^\perp. 
\]
There are several ways how to show (\ref{stab}), here
we make use of the following result
due to Arendt and Batty \cite{arendt:88}:

\begin{theorem}\label{thmArendtBatty}
Let $(T(t))_{t\geq0}$ be a bounded $C_0$-semigroup on a reflexive space $X$. Denote by $A$ the generator of $(T(t))$ and by $\sigma(A)$ the spectrum of $A$. If $\sigma(A)\cap i\mathbb{R}$ is countable and no eigenvalue of $A$ lies on the imaginary axis, then $\ds \lim_{t\rightarrow+\infty} T(t)x = 0$ for
all $x\in X$.
\end{theorem}

In view of this theorem we need to identify the spectrum of
${\mathcal A}_d$ lying on the imaginary axis.

First we look at the point spectrum.

\begin{lemma}\label{pointspectrum}
Suppose that $|\omega| > 0$.
If $\lambda$ is a non-zero real number, then $i\lambda$  is not an eigenvalue of ${\mathcal A}_d$.
\end{lemma}

\begin{proof}

Suppose that
\be\label{eigenir}
\left\{
\begin{array}{c}
i\lambda \,
\vec{u} + \nabla r + \alpha \, \vec{u}=\vec{0}, \\
i\lambda \, r + div \, \vec{u} =0.
\end{array}
\right.
\ee

From (\ref{eigenir}) we deduce that 
\[
i \lambda \int_{\Omega} (|\vec u|^2 - |r|^2 ) {\rm d}x + \int_{\Omega} \alpha |\vec u|^2 \ {\rm d}x = 0;
\]
whence $\vec u|_{{\rm supp} \alpha} = 0$, and, consequently, $\vec u$, $r$ solve (\ref{eigenir}) with $\alpha = 0$. 
In particular, we get 
\[
- \Delta r = \lambda^2 r , \ \nabla r \cdot n|_{\Gamma} = 0, \ r|_{{\rm supp} \alpha} = 0,
\]
and, by unique continuation for elliptic problems, we get $r = 0$.
\end{proof}

In accordance with Lemma \ref{pointspectrum} and the discussion in the previous section, 
$\lambda = 0$ is the only possibly eigenvalue of ${\cal A}_d$ on the imaginary axis. 

Next, we show that ${\mathcal A}_d$ has no continuous spectrum on the imaginary axis, except eventually zero.
\begin{lemma}\label{resolventiR}
Suppose that $|\omega| > 0$.
If $\lambda$ is a non-zero real number, then $i\lambda$ belongs to the resolvent set $\rho({\mathcal A}_d)$ of ${\mathcal A}_d$.
\end{lemma}
\begin{proof}
In view of Lemma \ref{pointspectrum} 
it is enough to show that
$i\lambda I-{\mathcal  A}_d$ is surjective.

Hence given  a vector $\left(\vec{f}, p\right)\in H$,
we look for $\left(\vec{u},r\right)\in{\mathcal D}({\mathcal
A}_d)$ such that
\be
\label{CE0i}(i\lambda I-{\mathcal  A}_d
)\left(\vec{u},r\right)=\left(\vec{f}, p\right). \ee
By the definition of ${\mathcal  A}_d$, we obtain
$$
\left\{
\begin{array}{c}
i\lambda \,
\vec{u} + \nabla r + \alpha \, \vec{u}=\vec{f}, \\
i\lambda \, r + div \, \vec{u} =p.
\end{array}
\right.
$$

Assuming that $\vec{u}$ and $r$ exist we can write
\be\label{defvecui}
\vec{u} =\frac{1}{i\lambda+\alpha}(-\nabla r+\vec{f}).
\ee
Inserting this expression in the second identity we obtain the differential equation in $r$:
\be\label{DEinri}
i\lambda \, r - \div \, \left(\frac{1}{i\lambda+\alpha}\nabla r\right) =p-\div \, \left(\frac{1}{i\lambda+\alpha} \vec{f}\right) \hbox{ in } \Omega, \ \nabla r \cdot n|_{\Gamma} =0.
\ee

Multiplying this identity by a test function $s\in H^1(\Omega)$, integrating in $\Omega$
and
using formal integration by parts we get the problem:
\be\label{DEinrweaki}
\int_{\Omega} \left(
i\lambda r\bar{s} + \frac{1}{i\lambda+\alpha}\nabla r\cdot \nabla \bar{s}\right)dx =F(s), \forall s\in H^1(\Omega),
\ee
where
$$
F(s)=\int_{\Omega} \left(p\bar{s}+\frac{1}{i\lambda+\alpha} \vec{f} \cdot \nabla \bar{s}\right)\, dx.
$$

We use the Fredholm alternative by splitting the left-hand side of \rfb{DEinrweaki} into
its principal part
$$
a_p(r,s)=\int_{\Omega}   \frac{1}{i\lambda+\alpha}\nabla r\cdot \nabla \bar{s} dx,
$$
and its lower order term
$$
a_0(r,s)= i\lambda \int_{\Omega}   r\bar{s} dx.
$$
The principal part is a continuous  sesquilinear coercive form on
$H^1_m(\Omega)$, 
\[
H^1_m(\Omega) = \{ v \in H^1(\Omega), \ \int_{\Omega} v \ {\rm d}x = 0 \};
\]
hence it induces an isomorphism $A_p$ from $H^1_m(\Omega)$ into $(H^1_m(\Omega))'$.
The mapping $A_0$ from $H^1_m(\Omega)$ into $(H^1_m(\Omega))'$ induces by $a_0$
is clearly given by $i\lambda I$ and is therefore compact.
Consequently by the Fredholm alternative,
$A_p+A_0$ is an isomorphism from $H^1_m(\Omega)$ into $(H^1_m(\Omega))'$ if and only if
it is injective.
But the injectivity of $A_p+A_0$ is equivalent to the injectivity of $i\lambda I-{\mathcal  A}_d$.
Indeed let $r\in H^1_m(\Omega)$ be such that $(A_p+A_0)r=0$ or equivalently
such that
\be\label{weakker}
\int_{\Omega} \left(
i\lambda r\bar{s} + \frac{1}{i\lambda+\alpha}\nabla r\cdot \nabla \bar{s}\right)dx =0, \forall s\in H^1_m(\Omega).
\ee
Since $r$ is of mean zero, this identity remains valid for all  $s\in H^1(\Omega)$. Hence by taking
$s\in \cD(\Omega)$ we find that
\be\label{DEker}
i\lambda \, r - \div \, \left(\frac{1}{i\lambda+\alpha}\nabla r\right) =0 \hbox{ in } \Omega.
\ee

By taking $s=r$ in \rfb{weakker}, we find
$$
\int_{\Omega} \left(
i\lambda |r|^2 + \frac{1}{i\lambda+\alpha}|\nabla r|^2\right)dx =0.
$$
By taking the real part of this identity, we find that
$$
\alpha |\nabla r|^2=0 \hbox{ in } \Omega.
$$
This implies that
$$
\nabla r=0 \hbox{ in } \omega.
$$
Hence by \rfb{DEker}, $r=0$ in $\omega$ and by the Holmgren uniqueness theorem we deduce that $r=0$.

In conclusion $A_p+A_0$ is an isomorphism from $H^1_m(\Omega)$ into $(H^1_m(\Omega))'$
which implies that there exists a unique solution $r\in H^1_m(\Omega)$ of
$$
\int_{\Omega} \left(
i\lambda r\bar{s} + \frac{1}{i\lambda+\alpha}\nabla r\cdot \nabla \bar{s}\right)dx =F(s), \forall s\in H^1_m(\Omega),
$$
As $F(1)=0$, we deduce that $r\in H^1_m(\Omega)$ is solution of  \rfb{DEinrweaki}
and relation \rfb{CE0i}
holds.
\end{proof}

These Lemmas and Theorem \ref{thmArendtBatty} leads to
\begin{corollary}
\label{cconv}
Let $(\vec u, r)$ be the unique semi-group solution of the problem (\ref{fluide}) emanating from the initial data 
$(\vec u^0,r^0) \in H$. Let $P_E$ be the orthogonal projection onto the space $E = {\rm Ker}[{\cal A}_d ]$ in $H$, and let 
\[
(\vec w,s) = P_E (\vec u^0, r^0).
\] 

Then
\[ 
\| (\vec u,r)(t, \cdot) - (\vec w, s) \|_{H} \to 0 \ \mbox{as}\ t \to \infty.
\]
\end{corollary}

\section{Exponential stability} \label{section4}
\setcounter{equation}{0}

Now, we may ask,
under some conditions on the damping coefficient, the convergence in Corollary \ref{cconv} is 
exponential in time.
\begin{theorem} \label{expstab}
If the damping coefficient $\alpha$ is not uniformly positive definite, meaning 
\[
\inf_{x \in \Omega} \alpha (x) = 0,
\]
then the system \rfb{fluide}
is not exponentially stable. Conversely, if $\alpha$ is uniformly positive definite, meaning 
\[
\omega = \Omega,
\]
then the system \rfb{fluide}
is exponentially stable, specifically, 
\[
\| (\vec u, r)(t, \cdot) \|_{H} \leq \exp(-Lt) \| (\vec u^0, r^0) \|_H ,\ L > 0, \ \mbox{whenever} \int_{\Omega} r^0 \ {\rm d}x = 0.
\]
\end{theorem}

\begin{proof}
According to \cite{haraux:89a} (see also \cite{ammari:01}),  the exponential stability of the system \rfb{fluide} is equivalent that the undamped system, i.e.
\be
\label{fluideb}
\left\{
\begin{array}{l}
\vec{\phi}_t + \nabla p  = 0, \,\hbox{ in } \Omega \times \rline^+, \\
p_t + div \vec{\phi} = 0, \,\hbox{ in } \Omega \times \rline^+, \\
\vec{\phi} \cdot n = 0, \,\hbox{ on } \Gamma \times \rline^+, \\
\vec{\phi}(0,x) = \vec{\phi}^0(x), \, p(0,x) = p^0(x), \, x \in \Omega,
\end{array}
\right.
\ee
satisfies the following inequality:
There exist positive real numbers $T,C$ such that
\be\label{obs1.1}
\int_0^T \int_\Omega \alpha(x) \, \left|\vec{\phi}(x,t) \right|^2 \, dxdt \geq \, C \, \left\|(\vec{\phi}^0,p^0)\right\|^2_{(L^2(\Omega)^d \times L^2(\Omega)}
\ee
\[
\forall \,(\vec{\phi}^0,p^0) \in (L^2(\Omega))^d \times L^2(\Omega)
\ \ \mbox{such that}\ 
\int_{\Omega} p^0 \ {\rm d}x = 0.
\]
Since this estimate is well-defined in the energy space, it holds if and only if it holds for strong solutions.

As the conservative system (\ref{obs1.1}) admits solutions that are constant in time, namely, 
\[
\vec u \in L^2(\Omega), \ \nabla \cdot \vec u |_{\Gamma} = 0,\ p = 0,
\]
it is clear that exponential stability cannot hold if $\inf_{x \in \Omega} \alpha (x) = 0$. Indeed, as $\alpha$ is smooth, 
we can always find for any $\delta > 0$ small a solenoidal compactly supported function $\vec \phi^0$ in $\Omega$ such that 
\[
\alpha|_{{\rm supp}\vec \phi^0} \leq \delta.
\]
Consequently, relation (\ref{obs1.1}) cannot holds \emph{uniformly} for any choice of the data. 

On the other hand, suppose that $\alpha$ is bounded below away from zero on the whole set $\Omega$. Writing 
the vector field $\vec u$ as its Helmholtz decomposition 
\begin{equation}\label{helmholtz}
\vec u = \vec{H}[\vec u] + \nabla \varphi,
\end{equation}
where $\vec{H}$ denotes the standard Helmholtz projection onto the space of solenoidal functions, it is enough to 
verify the observability criterion (\ref{obs1.1}) for $\vec \phi = \nabla \varphi$. In such a case, however, the 
conservative system (\ref{fluideb}) reduces to the standard \emph{wave} equation and (\ref{obs1.1}) is obviously 
satisfied as the damping acts uniformly on the whole domain $\Omega$ (see also Section \ref{section5}). 
\end{proof}

\section{Changing the damping law} \label{section5}
\setcounter{equation}{0}

As we have seen before, system \rfb{fluide} is exponentially stable if and only if $\alpha$ is uniformly positive definite. We will show in this section that if we change the feedback law in order to  filter
the divergence free vector fields, then we will get exponential  stability for a quite large set of $\alpha$.

In view of the  Helmholtz decomposition (\ref{helmholtz}),
denote by $P$ the orthogonal projection on the closed subspace of $L^2(\Omega)^d$
$$
V:=\left\{\nabla \varphi:\varphi\in H^1(\Omega) \int_{\Omega} \varphi \ {\rm d}x = 0 \right\}.
$$
Then in \rfb{fluide} we change the damping term $\alpha \, \vec{u}$ by
$P(\alpha \, P\vec{u})$ and consider the system

\be
\label{fluidemod}
\left\{
\begin{array}{l}
\vec{u}_t + \nabla r + P(\alpha \, P\vec{u}) = 0, \, \hbox{ in }\Omega \times \rline^+, \\
r_t + div \vec{u} = 0, \,\hbox{ in } \Omega \times \rline^+, \\
\vec{u} . n = 0, \, \hbox{ on }\Gamma \times \rline^+, \\
\vec{u}(0,x) = \vec{u}^0(x), \, r(0,x) = r^0(x), \, x \in \Omega.
\end{array}
\right.
\ee

Accordingly, we arrive at the system 
\be
\label{fluidemodnabla}
\left\{
\begin{array}{l}
\nabla \varphi_t + \nabla r + P(\alpha \, \nabla \varphi) = 0, \, \hbox{ in }\Omega \times \rline^+, \\
r_t + \Delta \varphi = 0, \,\hbox{ in } \Omega \times \rline^+, \\
\nabla \varphi \cdot n = 0, \, \hbox{ on }\Gamma \times \rline^+, \\
\varphi(0,x) = \varphi^0(x), \, r(0,x)=r^0(x), \, x \in \Omega,
\end{array}
\right.
\ee
where
\[
\int_{\Omega} \varphi^0 \ {\rm d}x = \int_{\Omega} r^0 \ {\rm d}x = 0.
\]

By virtue of  \cite{haraux:89a} (see also \cite{ammari:01}),  the exponential stability of the system \rfb{fluidemodnabla} on $V\times V$ is equivalent to the following property of solutions to the \emph{undamped} system:

There exist a positive real number $T_0$ such that for all $T>T_0$, there exists $C>0$ such that
\be\label{obs1.1mod}
\int_0^T \int_\Omega \alpha(x) \, \left|\nabla{\psi}(x,t) \right|^2 \, dxdt \geq \, C \, \left\|(\nabla \psi^0,p^0)\right\|^2_{(L^2(\Omega)^d \times L^2(\Omega)}
\ee 
for any $\psi$, $p$ satisfying 
\be
\label{fluidebmod}
\left\{
\begin{array}{l}
\nabla \psi_t + \nabla p  = 0, \,\hbox{ in } \Omega \times \rline^+, \\
p_t + \Delta \psi = 0, \, \hbox{ in }\Omega \times \rline^+, \\
\nabla \psi \cdot n = 0, \,\hbox{ on } \Gamma \times \rline^+, \\
{\psi}(0,x) = {\psi}^0(x),  \, p(0,x) = p^0(x), \, x \in \Omega,
\end{array}
\right.
\ee
\[
\psi^0 \in H^1(\Omega), \ p^0 \in L^2(\Omega), \ \int_{\Omega} \psi^0 \ {\rm d}x = \int_{\Omega} p^0 
{\rm d}x = 0.
\]

However, the system (\ref{fluidebmod}) can be written in the form of a standard \emph{wave equation} with 
the Neumann Laplacian: 
\be
\label{waveNeumannmod}
\left\{
\begin{array}{l}
\psi_{tt} + \Delta {\psi} = 0, \, \hbox{ in }\Omega \times \rline^+, \\
\nabla \psi \cdot n = 0, \, \hbox{ on } \Gamma \times \rline^+,\\
{\psi}(0,x) = \psi^0(x), \, {\psi}_t(0,x) = \psi^1(x)= -p^0(x), \, x \in \Omega,
\end{array}
\right.
\ee 
whereas the \emph{observability inequality} (\ref{obs1.1mod}) reduces to 
\be\label{obs1.1moda}
\int_0^T \int_\Omega \alpha(x) \, \left|\nabla{\psi}(x,t) \right|^2 \, dxdt \geq \, C \, \left\|(\nabla \psi^0, \psi^1) \right\|^2_{(L^2(\Omega)^d \times L^2(\Omega)}.
\ee 

Thus we have obtained the following exponential stability result:

\begin{theorem} \label{stabchangstab}
The system \rfb{fluidemod} is exponentially stable, meaning, 
\[
\| (\vec u, r)(t, \cdot) \|_H \leq \exp(-Lt) \| (\vec u^0, r^0) \|_H \ \mbox{for a certain}\ L > 0,
\]
whenever
\[
\vec u^0 = \nabla \varphi^0, \varphi^0 \in H^1(\Omega), \ \int_{\Omega} \varphi^0 \ {\rm d}x = 
\int_{\Omega} r^0 \ {\rm d}x  = 0,
\]
if and only if any solution $\psi$ of the wave equation \rfb{waveNeumannmod} satisfies the observability inequality \rfb{obs1.1moda}. 
\end{theorem}

Validity of the observability inequality (\ref{obs1.1moda}) is related to the general discussion in Bardos et al. \cite{BDL_92} (see also Zuazua \cite{zuazua:88,zuazua:91}). Here we note that (\ref{obs1.1moda}) holds
provided
$\omega$ contains a neighborhood  of the whole boundary $\Gamma$, in the sense that
there exists a  neighborhood   $O$ of $\Gamma$ in $\rline^d$ such that $\Omega\cap O\subset \omega$.
First we notice that by the arguments of section 4 of \cite{Martinez:99} with $a=0$ and $M(u)=m\nabla u+\frac{d-1}{2} u$ (see estimate (4.10)),
where,
as usual, $m(x)=x-x_0$ for some $x_0\in \rline^d$, there exits $C>0$ such that
for all $T>0$, we have
\be\label{estmartinez}
2T E(0)\leq C (E(0)+\int_0^T\int_{\Omega\cap O'}(|\psi_t|^2+|\nabla \psi|^2+|\psi|^2) dx dt,
\ee
where $O'$ is a sufficiently small neighbourhood of the boundary (such that
$\overline{O'}\subset O$). On the other hand,
fixing a cut-off function $\eta$ such that $\eta\equiv 1$ on $O'$
and with a support $O$, we can write
$$
\int_0^T\int_{\Omega\cap O'}|\psi_t|^2dx dt\leq \int_0^T\int_{\Omega} \eta|\psi_t|^2dx dt.
$$
Hence by integration by parts in time, we get
$$
\int_0^T\int_{\Omega\cap O'}|\psi_t|^2dx dt\leq -\int_0^T\int_{\Omega} \eta \psi_{tt} \bar \psi dx dt
+\int_{\Omega} \eta \psi_{t} \bar \psi dx\Big|_0^T.
$$
Now using the first identity in \rfb{waveNeumannmod} and an integration by parts in space, we get
$$
\int_0^T\int_{\Omega\cap O'}|\psi_t|^2dx dt\leq \int_0^T\int_{\Omega}   \nabla \psi\cdot  (\nabla \eta \bar \psi) dx dt
+\int_{\Omega} \eta \psi_{t} \bar \psi dx\Big|_0^T.
$$
Therefore using Leibniz's rule and Cauchy-Schwarz's inequality we obtain that
$$
\int_0^T\int_{\Omega\cap O'}|\psi_t|^2dx dt\leq C_1\int_0^T\int_{\Omega\cap O}   (|\nabla \psi|^2+|\psi|^2)dx dt
+C_2 E(0),
$$
for some positive constants $C_1, C_2$ independent of $T$.
Inserting this estimate in \rfb{estmartinez} we get
\be\label{estmartinez2}
2T E(0)\leq (C+C_2)E(0)+C_4\int_0^T\int_{\Omega\cap O}(|\nabla \psi|^2+|\psi|^2) dx dt,
\ee
where $C_4=\max\{1, C_1\}$.
To eliminate the last term of this right-hand side we use a compacteness/uniqueness argument.
Namely  assume that \rfb{obs1.1mod} does not hold.
Then there exists a sequence of $(\psi_\ell)_{\ell\in \nline}$ solution of \rfb{waveNeumannmod} with initial data
$\psi^0_\ell$ and $\psi^1_\ell$ such that
\be\label{energy1}
\|\nabla \psi^0_\ell\|+\|\psi^1_\ell\|=1,
\ee
and
\be\label{intterm}
\int_0^T \int_\Omega \alpha(x) \, \left|\nabla{\psi_\ell}(x,t) \right|^2 \, dxdt=\frac1\ell.
\ee
Since the system \rfb{waveNeumannmod} is conservative, the sequence
$(\psi_\ell)$ is bounded in $H^1(Q_T)$, where $Q_T=\Omega\times (0,T))$. Hence it converges strongly
in $L^2(Q_T)$ to $\psi$ a weak solution (in $H^1(Q_T)$ of \rfb{waveNeumannmod} with initial data $\psi^0, \psi^1$. But thanks to \rfb{estmartinez2}
applied to $\psi_\ell-\psi_{\ell'}$, we deduce that the sequence $(\psi_\ell)_{\ell\in \nline}$ is a Cauchy sequence in $H^1(Q_T)$. Hence $\psi_\ell$ converges in $H^1(Q_T)$ to $\psi$, weak solution of \rfb{waveNeumannmod} and satisfying
\be\label{energy1psi}
\|\nabla \psi^0\|+\|\psi^1\|=1,
\ee
and
\be\label{inttermpsi}
\int_0^T \int_\Omega \alpha(x) \, \left|\nabla{\psi}(x,t) \right|^2 \, dxdt=0.
\ee
Accordingly
$$\nabla \psi=0 \hbox{ in } \omega\times (0,T),
$$
or quivalently
$$
\psi(x,t)=\psi(t) \hbox{ in } \omega\times (0,T).
$$
But again due to the first identity in \rfb{waveNeumannmod},
$$
\psi_{tt}(t)=0 \hbox{ in } \omega\times (0,T),
$$
hence there exists two complex numbers $a,b$ such that
$$
\psi(x,t)=a t+b \hbox{ in } \omega\times (0,T).
$$
Now consider the difference
$$
\tilde \psi=\psi-at-b.$$ Then we see that it a weak solution of \rfb{waveNeumannmod}
such that
$$
\tilde \psi=0 \hbox{ in } \omega\times (0,T).
$$
By using Theorem 9.1 of \cite{Triggiani-Yao_02}, we deduce that
$$
\tilde \psi=0 \hbox{ on } Q_T,
$$
or equivalently
$$
\psi(x,t)=a t+b \hbox{ on } Q_T.
$$
But in our situation $\psi^0$ belongs to $H^1_m(\Omega)$
and therefore $\psi(\cdot, t)$ belongs to $H^1_m(\Omega)$, for all $t>0$.
As a consequence we deduce that
$$
0=\int_\Omega\psi(x,t)\,dx=(a t+b)|\Omega|, \forall t>0.
$$
This imples that $a=b=0$ and contradicts \rfb{energy1psi}.

Thus we have shown the following:

\begin{corollary} \label{CCC1}
Suppose that
\[
\inf_{x \in \Gamma} \alpha(x) > 0.
\]

Then the system (\ref{fluidemod}) is exponentially stable in the sense specified in Theorem 
\ref{stabchangstab}.
\end{corollary}

\section{Polynomial stability} \label{section6}
\setcounter{equation}{0}

In this section we show that system \rfb{fluide} is polynomially stable
if $\omega$ contains a neighborhood  of the whole boundary $\Gamma$ as in Corollary \ref{CCC1}.

This result is based on  the following result
stated in Theorem 2.4 of \cite{borichev:10} (see also
\cite{batkai:06,batty:08,Liu_Rao_05} for weaker variants).
\begin{lemma}\label{lemrao}
A $C_0$ semigroup $e^{t{\mathcal L}}$ of contractions on a Hilbert
space such that
\begin{equation}
\rho ({\mathcal L})\supset \bigr\{i \beta \bigm|\beta \in \mathbb{R}
\bigr\} \equiv i \mathbb{R}, \label{1.8w} \end{equation}
satisfies
$$||e^{t{\mathcal L}}U_0|| \leq C \, {t^{-\frac{1}{l}}} ||U_0||_{{\mathcal D}({\mathcal L})},\quad
 \forall U_0\in\mathcal{D}(\mathcal{L}),\quad\forall t>1,$$
 as well as
$$
||e^{t{\mathcal L}}U_0|| \leq C \, t^{-1} ||U_0||_{{\mathcal
D}({\mathcal L}^l)},\quad
 \forall U_0\in\mathcal{D}(\mathcal{L}^l),\quad\forall t>1,$$
for some constant $C >0$ and for some positive integer $l$ if and only if
\begin{equation} \limsup_{|\beta |\to \infty } \frac{1}{\beta^l} \,
\|(i\beta -{\mathcal L})^{-1}\| <\infty. \label{1.9}
\end{equation}
\end{lemma}

\begin{lemma}\label{lemresolvent} Assume that $\omega$ contains a neighborhood  of the whole boundary $\Gamma$.

Then the resolvent of the operator of $\mathcal{A}_d$ satisfies
condition (\ref{1.9}) with $l=3$.
\end{lemma}
\begin{proof}
We use a contradiction argument, i.e., we suppose that
(\ref{1.9}) is false with some $l\geq 0$. Then there exist a sequence of real numbers
$\beta_n\rightarrow+\infty$ and a sequence of vectors
$Z_n=(\vec{u}_n,r_n)^\top$ in $\mathcal{D}(\mathcal{A}_d)$ with
\be\label{znborne}
\left\|Z_n\right\|_{\mathcal{H}}=1,
\ee and
\begin{equation}\label{serge99}
\beta_n^l\left\|(i\beta_n-\mathcal{A}_d)Z_n\right\|_{\mathcal{H}}\rightarrow 0
\hbox{ as } n \rightarrow \infty.
\end{equation}
From the definition of $\mathcal{A}_d$ this last property is equivalent to
\be\label{1}
\beta_n^l(i\beta_n \vec{u}_n+\nabla r_n+\alpha \vec{u}_n)=: f_n\to 0\hbox{ in } L^2(\Omega)^d,
\ee
and
\be\label{2}
\beta_n^l(i\beta_n r_n+\div \vec{u}_n)=: g_n\to 0\hbox{ in } L^2(\Omega).
\ee

As usual, taking the inner product of \rfb{serge99} with $Z_n$, using \rfb{znborne}
and the dissipativeness of $\mathcal{A}_d$, we find
\be\label{3}
\int_{\Omega} \alpha \, |\vec{u}_n|^2  dx=\Re ((i\beta_n-\mathcal{A}_d)Z_n, Z_n)=o(\beta_n^{-l}),
\ee
where, here and hereafter, 
$a_n=o(b_n)$ means that
$$
\lim_{n\to\infty} \frac{a_n}{b_n}=0.
$$

We multiply \rfb{1} by $\vec{u}_n$, integrate in $\Omega$  and use \rfb{znborne} to find
$$
\int_\Omega (i\beta_n |\vec{u}_n|^2+\nabla r_n\cdot \vec{u}_n+\alpha |\vec{u}_n|^2)\,dx=o(\beta_n^{-l}).
$$
Using \rfb{3} and integrating by parts we obtain
$$
\int_\Omega (i\beta_n |\vec{u}_n|^2- r_n\div \vec{u}_n)\,dx=o(\beta_n^{-l}).
$$
Using \rfb{2} (and \rfb{znborne}) we get
$$
i\beta_n\int_\Omega (|\vec{u}_n|^2- |r_n|^2)\,dx=o(\beta_n^{-l}).
$$
This shows that
\be\label{4}
\|\vec{u}_n\|^2- \|r_n\|^2=o(\beta_n^{-(l+1)}).
\ee

The identity \rfb{1} implies that
\be\label{5}
\nabla r_n=\beta_n^{-l} f_n- (i\beta_n +\alpha) \vec{u}_n,
\ee
and therefore
\be\label{7}
\|\nabla r_n\|=O(\beta_n),
\ee
where, here and hereafter, $a_n=O(b_n)$
means that there exists $C>0$ independent of $n$ such that
$a_n\leq C b_n$ for $n$ large enough.

By \rfb{5}, we also have
$$
\int_{\Omega} \alpha \, |\nabla r_n|^2  dx\leq 2 \int_{\Omega} \alpha |i\beta_n +\alpha|^2 \, |\vec{u}_n|^2  dx+2 \beta_n^{-2l}\int_{\Omega} \alpha  \, |\vec{f}_n|^2  dx.
$$
Hence, for $l\geq 2$, by using \rfb{3} we find that
\be\label{7ter}
\int_{\Omega} \alpha \, |\nabla r_n|^2  dx=o(\beta_n^{2-l}).
\ee

Now we notice that \rfb{1} and \rfb{3} imply that
\be\label{8}
i\beta_n \vec{u}_n+\nabla r_n=\vec{h}_n:=\beta_n^{-l} f_n-\alpha \vec{u}_n,
\ee
with
\be\label{9}
\|\vec{h}_n\|=o(\beta_n^{-l/2}).
\ee

We use the Helmholtz decomposition
\be\label{10}
\vec{h}_n=\nabla \varphi_n+ \vec \chi_n,
\ee
$$
\| \nabla \varphi_n \|_{\Omega}^2+\|\vec \chi_n\|_{\Omega}^2=\|\vec{h}_n\|^2
$$
to deduce that
\be\label{11}
\| \varphi_n \|_{1,\Omega}^2+\|\vec \chi_n\|_{\Omega}^2=o(\beta_n^{-l}).
\ee

The identities \rfb{8} and \rfb{10} yield
$$
i\beta_n \vec{u}_n+\nabla r_n=\nabla \varphi_n+ \vec \chi_n,
$$
or equivalently
\be\label{12}
i\beta_n \vec{u}_n=-\nabla s_n+\vec \chi_n,
\ee
with
\be\label{12bis}
s_n=r_n-\varphi_n.
\ee

As $\vec{u}_n$ is solenoidal, $s_n$ is harmonic, and, 
furthermore since $\vec{u}_n\cdot n=\vec \chi_n \cdot n=0$ on $\Gamma$, we automatically have
\be\label{12bis}
\nabla s_n \cdot n=0 \hbox{ on } \Gamma.
\ee
Next, by \rfb{znborne} and \rfb{11}, one has
\be\label{14}
\|s_n\|=O(1).
\ee
while \rfb{7} and \rfb{11} lead to
\be\label{14bis}
\|\nabla s_n\|=O(\beta_n).
\ee

Now by \rfb{2}, \rfb{12}  and \rfb{12bis} we see that $s_n$ satisfies
\be\label{15}
i\beta_n s_n-\frac{1}{i\beta_n}\div \nabla s_n=\tilde g_n
\ee
where
$$
\tilde g_n:=\beta_n^{-l} g_n-i\beta_n \varphi_n.
$$
Due to
\rfb{2} and \rfb{11} we see that
\be\label{16}
\|\tilde g_n\|=o(\beta_n^{1-l/2}).
\ee

At this stage we use the multiplier method to find some properties on $s_n$.
First we multiply \rfb{15} by $-i\bar s_n$ and integrate in $\Omega$
to obtain due to \rfb{14} and \rfb{16}
$$
\int_\Omega (\beta_n s_n+\frac{1}{\beta_n}\div \nabla s_n) \bar s_n\,dx=o(\beta_n^{1-l/2}).
$$
By an integration by part  and since   the boundary term is zero due to \rfb{12bis}, we obtain
\be\label{17}
\int_\Omega (\beta_n |s_n|^2-\frac{1}{\beta_n}|\nabla s_n|^2)\,dx=o(\beta_n^{1-l/2}).
\ee

Secondly we take a (smooth) real-valued multiplier $m\in C^2(\Omega)^d$ fixed later on but such that
$m=0$ on $\Gamma$
and multiply
\rfb{15} by $-im\cdot \nabla\bar s_n$ and integrate in $\Omega$
to obtain due to \rfb{14bis} and \rfb{16}
\be\label{18}
\int_\Omega (\beta_n s_n+\frac{1}{\beta_n}\div \nabla s_n) m\cdot \nabla\bar s_n\,dx=o(\beta_n^{2-l/2}).
\ee

Now,
in a standard way, by means of Green's formula, the first term
$$
I_1:=\int_\Omega   s_n m\cdot \nabla\bar s_n\,dx
$$
is transformed into
\beqs
I_1&=&-\int_\Omega   \partial_k (s_n m_k) \bar s_n\,dx+\int_\Gamma  m\cdot n |s_n|^2d\sigma\\
&=&-\bar I_1-\int_\Omega   \div m |s_n|^2\,dx,
\eeqs
recalling that $m=0$ on the boundary. Hence we have
\be\label{19}
2 \Re I_1=-\int_\Omega   \div m |s_n|^2\,dx.
\ee
In the same manner we have
\begin{eqnarray}\label{19bis}
I_2:=\int_\Omega\div \nabla s_n m\cdot \nabla\bar s_n\,dx&=&
-\int_\Omega  \partial_k s_n \partial_k(m\cdot \nabla\bar s_n)\,dx
\\
\nonumber
&=&
-\int_\Omega  \partial_k s_n \partial_k m_j \partial_j\bar s_n\,dx
-I_3,
\end{eqnarray}
where
$$
I_3=\int_\Omega  \partial_k s_n   m_j \partial_k \partial_j\bar s_n\,dx.
$$
Again an integration by part leads to

\beqs
I_3&=&-\int_\Omega  \partial_j(\partial_k s_n   m_j) \partial_k \bar s_n\,dx
\\
&=&-\bar I_3-\int_\Omega   \div m |\nabla s_n|^2\,dx.
\eeqs
Consequently
$$
2\Re I_3=-\int_\Omega   \div m |\nabla s_n|^2\,dx,
$$
and using this expression in \rfb{19bis} we find

\begin{eqnarray}\label{20}
2\Re I_2&=&
-2\Re \int_\Omega  \partial_k s_n \partial_k m_j \partial_j\bar s_n)\,dx
\\
&+&\int_\Omega   \div m |\nabla s_n|^2\,dx.
\nonumber
\end{eqnarray}
Taking the real part of \rfb{18} and using \rfb{19} and \rfb{20} we obtain
\begin{eqnarray}\label{21}
&&-\beta_n\int_\Omega   \div m |s_n|^2\,dx
\nonumber
+\frac{1}{\beta_n}\int_\Omega   \div m |\nabla s_n|^2\,dx
\\
&&-\frac{2}{\beta_n} \Re \int_\Omega  \partial_k s_n \partial_k m_j \partial_j\bar s_n\,dx
=o(\beta_n^{2-l/2}).
\end{eqnarray}

From \rfb{7ter} and \rfb{11}, we see that
\be\label{7quatro}
\int_\Omega \alpha |\nabla s_n|^2 \,dx=o(\beta_n^{2-l}),
\ee
hence for $l>2$, $\nabla s_n$ tends to zero on $\omega$. This means that we are mainly interested in the behaviour of $s_n$ outside $\omega$.
Therefore we take for $m$ a function with a support far from the boundary, namely we take
$$
m(x)=\eta(x) x,
$$
where $\eta$ is a smooth cut-off function such that
$$
\eta=1 \hbox{ on } \omega^c\quad \hbox{ and } \eta=0 \hbox{ in } O',
$$
where $O'$ is a neigbourhood of $\Gamma$. Since
$$
\partial_k m_j=\delta_{jk} \eta+x_j \partial_k \eta,
$$
\rfb{21} becomes
\begin{eqnarray}\label{21bis}
&&-\beta_n\int_\Omega   (d\eta + x\cdot \nabla \eta) |s_n|^2\,dx
\nonumber
+\frac{1}{\beta_n}\int_\Omega   ((d-2)\eta + x\cdot \nabla \eta) |\nabla s_n|^2\,dx
\\
&&-\frac{2}{\beta_n} \Re \int_\Omega  \partial_k s_n \partial_k \eta x_j \partial_j\bar s_n\,dx
=o(\beta_n^{2-l/2}).
\end{eqnarray}

Assume that the next estimate
\be\label{24}
\left(\int_\omega |s_n|^2 \,dx\right)^\frac12=o(\beta_n^{1-l/2})
\ee
holds. Then combining this estimate with   \rfb{7quatro} in \rfb{21bis} we obtain
\be\label{25}
-\beta_n d\int_\Omega   \eta  |s_n|^2\,dx
+\frac{1}{\beta_n} (d-2)\int_\Omega   \eta   |\nabla s_n|^2\,dx
=o(\beta_n^{2-l/2}).
\ee

Coming back to \rfb{17} and writting $1=\eta+(1-\eta)$, we get
\beqs
\int_\Omega (\beta_n \eta|s_n|^2-\frac{1}{\beta_n}\eta|\nabla s_n|^2)\,dx
&=&o(\beta_n^{1-l/2})
\\
&-&\int_\Omega (\beta_n (1-\eta)|s_n|^2-\frac{1}{\beta_n}(1-\eta)|\nabla s_n|^2)\,dx.
\eeqs
Hence \rfb{24} and \rfb{7quatro} lead to

\be\label{26}
\int_\Omega (\beta_n \eta|s_n|^2-\frac{1}{\beta_n}\eta|\nabla s_n|^2)\,dx
=o(\beta_n^{1-l/2})+o(\beta_n^{3-l}).
\ee

This estimate in \rfb{25} yields
$$
  \int_\Omega   \eta  |s_n|^2\,dx
=o(\beta_n^{1-l/2})+o(\beta_n^{2-l}).
$$
By choosing $l\geq 2$ and using again \rfb{24}, we arrive at
\be\label{27}
\int_\Omega     |s_n|^2\,dx
=o(\beta_n^{1-l/2}).
\ee

Taking into account \rfb{17} we get
$$
\int_\Omega |\nabla s_n|^2\,dx=
\beta_n^2\int_\Omega  |s_n|^2\,dx+o(\beta_n^{2-l/2}),
$$
and   \rfb{27}
finally leads to
\be\label{28}
\int_\Omega |\nabla s_n|^2\,dx=o(\beta_n^{3-l/2}).
\ee

Coming back to our original variables, for $r_n$ we have by   \rfb{27} and \rfb{11}
that
$$
\|r_n\|^2=o(\beta_n^{1-l/2}),
$$
while for $\vec{u}_n$ using \rfb{12}, \rfb{28} and \rfb{11}
$$
\|\vec{u}_n\|^2=o(\beta_n^{1-l/2}).
$$
In conclusion for $l\geq 3$, $\|r_n\|$ and $\|\vec{u}_n\|$ tend to zero
which contradicts \rfb{znborne}.

It remains to prove \rfb{24}.
For that purpose, we first show that the mean of $s_n$ on $\Gamma$ tends to zero.
Fix a function $h\in (C^2(\bar\Omega))^d$  such that
supp $h\subset \omega$ and
$$
h\cdot n=1 \hbox{ on } \Gamma.
$$
Such a function exists by Lemma I.3.1 of \cite{Li} since we assume that $\Omega$ has a $C^3$-boundary.
We then multiply \rfb{1} by $h$ and integrate in $\Omega$ to get
$$
\int_\Omega ((i\beta_n+\alpha) \vec{u}_n+\nabla r_n)h\,dx=o(\beta_n^{-l}).
$$
Hence integrating by parts, we obtain
\be\label{serge1506}
\int_\Omega ((i\beta_n+\alpha) \vec{u}_n h-r_n \div h)\,dx
+\int_\Gamma r_n \,d\sigma =o(\beta_n^{-l}).
\ee
We still need to transform the term
$$
\int_\Omega r_n \div h\,dx.
$$
By using \rfb{2} we may write
$$\int_\Omega r_n \div h\,dx
=-i\int_\Omega (\beta_n^{-(l+1)} g_n-\div \vec{u}_n)\div h\,dx.
$$
Again by Green's formula and the fact that $\vec{u}_n\cdot n=0$ on $\Gamma$ we obtain
$$\int_\Omega r_n \div h\,dx
=-i\int_\Omega (\beta_n^{-(l+1)} g_n \div h+\vec{u}_n\cdot \nabla \div h) \,dx.
$$
Using this identity in \rfb{serge1506} we obtain
$$
\int_\Gamma r_n \,d\sigma =
-\int_\Omega ((i\beta_n+\alpha) \vec{u}_n h-i\beta_n^{-(l+1)} g_n \div h-i\vec{u}_n\cdot \nabla \div h)\,dx
+o(\beta_n^{-l}).
$$
Using \rfb{3} we get
\be\label{meanr}
\int_\Gamma r_n \,d\sigma =o(\beta_n^{1-l/2}).
\ee
But a standard Poincar\'e type inequality implies that
$$
\left(\int_\omega |r_n|^2 \,dx\right)^\frac12\leq C
(\left(\int_\omega |\nabla r_n|^2 \,dx\right)^\frac12+\left|\int_\Gamma r_n \,d\sigma\right|,
$$
for some $C>0$ depending only on $\omega$.
Hence using \rfb{meanr}
and   \rfb{7ter} yields
\be\label{23}
\left(\int_\omega |r_n|^2 \,dx\right)^\frac12=o(\beta_n^{1-l/2}).
\ee
This estimate and \rfb{11} leads to \rfb{24}.
\end{proof}

This Lemma and Lemmas \ref{pointspectrum} and \ref{resolventiR} show that the restriction of the operator  $\mathcal{A}_d$
to $H_0$ satisfies the hypotheses of Lemma \ref{lemrao}
with $l=3$. Therefore we have obtained the next main result.

\begin{theorem} \label{stabpoly}
Assume that $\omega$ contains a neighborhood  of the whole boundary $\Gamma$ then the system \rfb{fluide}
with an initial datum $(\vec{u}^0,r^0)$ in $H_0\cap D(\mathcal{A}_d)$, 
$H_0 = {\rm Ker}[{\cal A}_d]^\perp$, is polynomially stable, namely
there exists $C>0$ such that
$$
E(t)\leq C t^{-2/3} \|(\vec{u}^0,r^0)\|_{D(\mathcal{A}_d)}^2, \forall t>0.
$$
\end{theorem}

\protect\bibliographystyle{abbrv}
    \protect\bibliographystyle{alpha}
    \bibliography{/Users/sergenicaise/Documents/Serge/Desktop/Biblio/control,/Users/sergenicaise/Documents/Serge/Desktop/Biblio/valein,/Users/sergenicaise/Documents/Serge/Desktop/Biblio/kunert,/Users/sergenicaise/Documents/Serge/Desktop/Biblio/est,/Users/sergenicaise/Documents/Serge/Desktop/Biblio/femaj,/Users/sergenicaise/Documents/Serge/Desktop/Biblio/mgnet,/Users/sergenicaise/Documents/Serge/Desktop/Biblio/dg,/Users/sergenicaise/Documents/Serge/Desktop/Biblio/bib,/Users/sergenicaise/Documents/Serge/Desktop/Biblio/maxwell,/Users/sergenicaise/Documents/Serge/Desktop/Biblio/bibmix,/Users/sergenicaise/Documents/Serge/Desktop/Biblio/cochez,/Users/sergenicaise/Documents/Serge/Desktop/Biblio/soualem,/Users/sergenicaise/Documents/Serge/Desktop/Biblio/nic}

\end{document}